\documentclass[12pt]{amsart}
\usepackage[utf8]{inputenc}
\usepackage{amsmath}
\usepackage{amssymb}
\usepackage{amsfonts}
\usepackage{amsthm}
\usepackage{mathrsfs} 
\usepackage{xcolor}
\usepackage{hyperref}
\usepackage{xparse}
\usepackage{bm}
\usepackage{bbm}
\usepackage{tikz}
\usepackage[margin=1.3in]{geometry}
\usepackage{graphicx}
\usepackage[11pt]{moresize}
\usepackage[foot]{amsaddr}

\newcommand{\vs}{\vspace{3mm}}

\newcommand{\Z}{\mathbb{Z}}

\newcommand{\N}{\mathbb{N}}

\newcommand{\F}{\mathbb{F}}
\newcommand{\BB}{\mathcal{B}}

\newcommand{\ep}{\epsilon}

\DeclareMathOperator{\im}{Im}

\DeclareMathOperator{\Int}{int}

\DeclareMathOperator{\supp}{Supp}
\DeclareMathOperator{\Supp}{Supp}
\DeclareMathOperator{\dom}{Dom}
\DeclareMathOperator{\aux}{aux}

\hypersetup{
    colorlinks=true,
    linkcolor=blue,
    filecolor=magenta,
    urlcolor=blue,
    citecolor=blue
}

\newtheorem{thm}{Theorem}

\newtheorem{result}{Result}[section]
\newtheorem{lem}[result]{Lemma}
\newtheorem{rmk}[result]{Remark}

\newtheorem{cor}[result]{Corollary}

\newtheorem{prp}[result]{Proposition}

\theoremstyle{definition}
\newtheorem*{defn}{Definition}

\theoremstyle{remark}

\DeclareMathOperator{\Leaves}{\bm{\mathsf{Leaves}}}
\DeclareMathOperator{\leaves}{\bm{\mathsf{Leaves}}}
\DeclareMathOperator{\Inner}{\bm{\mathsf{Inner}}}
\DeclareMathOperator{\inner}{\bm{\mathsf{Inner}}}

\newcommand{\hide}[1]{}
\newcommand{\edit}[1]{\textcolor{red}{#1}}

\newcommand{\rough}[1]{}
\definecolor{darkgreen}{RGB}{75,150,75}
\newcommand{\review}[1]{}
\newcommand{\dc}[1]{}
\newcommand{\zh}[1]{}
\newcommand{\hides}[1]{}
\newcommand{\pub}[1]{}

\title{Optimally Reconstructing Caterpillars}

\author{Zach Hunter}
\email{zachary.hunter@exeter.ox.ac.uk}
\date{\today}

\begin{document}

\begin{abstract}
    For a graph $G$, the $\ell$-deck of $G$ is the multiset of induced subgraphs on $G$ having $\ell$ vertices. Recently, Groenland et al. proved that any tree can be reconstructed from its $(8/9+o(1))n$-deck. For the particular case of caterpillar graphs, we show that the $(1/2+o(1))n$-deck suffices, which is asymptotically tight.
\end{abstract}

\maketitle

\section{Introduction}\label{intro}

All graphs in this paper are finite and simple. Given a graph $G$, we define its $\ell$-deck, $\mathcal{D}_\ell(G)$ to be the multiset of subgraphs of $G$ induced by sets of $\ell$ vertices. We say a graph $G$ can be reconstructed by its $\ell$-deck if $\mathcal{D}_\ell(H) = \mathcal{D}_\ell(G)$ implies $H \cong G$.

\vs

The standard graph reconstruction conjecture claims that any graph with $n\ge 3$ vertices can be reconstructed by its $(n-1)$-deck. Nydl \cite{nydl} has shown that for any $\ep> 0$ and any integer $n_0$, there is $n>n_0$ and distinct graphs $A,B$ on $n$ vertices such that $\mathcal{D}_{(1-\ep)n}(A) = \mathcal{D}_{(1-\ep)n}(B)$, thus for general graphs this conjecture (if true) is asymptotically tight.

\vs

However, recently, in a paper by Groenland, Johnston, Scott, and Tan it has been shown that if $G$ is a tree on $n$ vertices, then it can be reconstructed by its $(8/9 +o(1))n$-deck \cite[Theorem~3]{deg}. Thus, for smaller classes of graphs, we can improve the linear coefficient. 

\vs

While the improved upper bound of Groenland et al. is impressive, it is not believed to be asymptotically optimal. In \cite[Problem~1]{deg} it is asked whether the $(1/2+o(1))n$-deck suffices to reconstruct $n$-vertex trees, and it is mentioned that the $(\lfloor n/2\rfloor +1)$-deck may even suffice for all sufficiently large $n$. The latter bound would be best possible, due to examples coming from the family of  ``caterpillar graphs''.

\vs

A \textit{caterpillar graph} is a tree such that the removal of all leaves results in a path. For every $n\ge 4$, it is known that there exist distinct $n$-vertex caterpillar graphs $A,B$ where $\mathcal{D}_{\lfloor n/2\rfloor}(A) = \mathcal{D}_{\lfloor n/2 \rfloor}(B)$. In this paper, we prove this lower bound for reconstructing caterpillar graphs is tight up to an additive constant.

\vspace{1.5mm}

\begin{thm}\label{main}Let $G$ be a caterpillar graph on $n$ vertices. Then $G$ can be reconstructed from its $(n/2 + O(1))$-deck.
\end{thm}

\noindent In fact, the proof of the theorem shows this reconstruction can be done in a particularly strong sense. Namely, that after recognizing $G$ is a tree with diameter $k$, we only need to look at cards with diameter at most $\lfloor k/2\rfloor + 1$ and $o(n-k)$ leaves.

\vs 

We now recall the construction of distinct caterpillar graphs $G_1,G_2$ on $n$ vertices, such that $\mathcal{D}_{\lfloor n/2\rfloor}(G_1) = \mathcal{D}_{\lfloor n/2\rfloor}(G_2) $. Essentially, we construct $G_1,G_2$ respectively by starting with a path on $n-1$ vertices, and then attaching a leaf to a ``central'' vertex in the path (respectively to a vertex which is not central but neighbors a central vertex). This example was observed by N\'ydl \cite{nydltree}. Said example (and generalizations where you have many central leaves for smaller choices of $k$) also demonstrate that considering cards of diameter $\lfloor k/2\rfloor +1$ is necessary to recognize some caterpillars with diameter $k$ (hence this aspect of our result is completely optimal).

\vs

In Section~\ref{preliminaries}, we  establish our definitions and some terminology. Then in Subsection~\ref{summary} we outline our reconstruction procedure, which will be divided into three parts. At its heart, our methods are a more precise utilization of $H$-extensions (a concept introduced in \cite[Section~5.1]{deg}), where we figured out how to efficiently reconstruct a special class of $H$-extensions and then realized how to reconstruct caterpillars with this special class (which carried much less information than the $H$-extensions considered in \cite{deg}).

\vs

In Section~\ref{mom}, we establish a key lemma about equalities between ``moments''. It is essentially a multivariable analogue of \cite[Lemma~11]{deg}. Our assumptions are a bit stronger, which allows for an elementary number theoretic proof. We believe this section may be of independent interest.

\vs

In Section~\ref{reconstructing caterpillars}, we apply this lemma about moments to turn the problem of reconstructing caterpillars into a problem of reconstructing binary strings up to reversal. In Section~\ref{binary strings}, we resolve the problem of reconstructing binary strings.

\vs

We mention that it should be possible to extend our methods to efficiently reconstruct other classes of graphs besides caterpillars. \hide{This is discussed further in Subsection~\ref{generalize}.}

\section{Preliminaries}\label{preliminaries}

\hide{\subsection{Draft 2}

We shall be generalizing the notion of degree sequences in graphs.

A colored graph $C$ is a pair $(G,c)$, where $G = (V,E)$ is a graph and $c:V\to \N$ is a map. A partially colored graph is a triple $(G,V',c)$, where $G = (V,E)$ is graph, $V' \subset V$, and $c:V'\to \N$ is a map. For a colored graph $C = (G,c)$, and $V' \supset V''$, we define $C[V',V'']$ to be the partially colored graph $(G[V'],V'',c|_{V''})$. 

For a colored graph $C$, we let $\mathscr{D}_{\ell,s}(C)$ (the ``deck'' of $C$) denote the multiset of unlabelled partially colored graphs $C[V',V'']$, where $|V'| = \ell, |V''| = s$. 

For a graph $G$, we define $\leaves(G)$ to be the set of $v\in G$ where $\deg(v) = 1$.

For a graph $G$, we let $\varphi_0(G) = (G,c)$ denote the colored graph where $c:v\mapsto \deg(v)$. It is clear that $\varphi_0$ is an injection. We then define the ``pruning'' of $G$, $\varphi(G)$, to be $\varphi_0(G)[V',V']$, where $V' = V(G)\setminus \Leaves(G) = \{v \in G: \deg(v) > 1\}$. We note that if $\varphi$ is non-empty (i.e., not all vertices have degree $<2$), then it is also an injection, since the coloring will ``remember'' all the low-degree vertices which are removed.

For a caterpillar graph $G$ with diameter $k+2$, we have that $\varphi(G)$ will be a colored path with diameter $k$. In Section~\ref{binary strings}, we show that a colored path of length $k$ can be reconstructed from

\subsection{Definitions}}

We shall always consider 0 to be an element of $\N$. We will write $[n]$ to denote $\{1,2,\dots, n\}$, and $[a,b]$ to denote $\{a,a+1,\dots, b\}$. For a set $S$, we write $i+S$ to denote $\{i+s: s \in S\}$. Also, for an event $E$, we will write $I(E)$ to denote the indicator function of $E$.

\vs

We will actually need very little notation about graphs. We will use just a few notations from \cite{deg}. The diameter of graph $G$ is the maximum distance between two vertices $u,v \in V(G)$, when $G$ is a tree this is the same as the number of edges in the longest path in $G$. We let $n_H(G)$ count the number of times a certain graph $H$ appears as a subgraph of $G$ (i.e. the number of subsets $U$ of $V(G)$ where $G[U] \cong H$). \begin{rmk}\label{kelly}(``Kelly's lemma'') If $V(H) \le \ell$, then we can calculate $n_H(G)$ given $\mathcal{D}_\ell(G)$, see Lemma~8 of \cite{deg}.\end{rmk} \begin{rmk}\label{dseq} We can reconstruct the degree sequence of $G$ from $\mathcal{D}_{\sqrt{2n \log(2n)}}(G)$, see Theorem~7 of \cite{deg}.\end{rmk} 

\vs

Lastly, for a graph $G$, we define the interior of $G$, $\inner(G)$, to be the graph remaining after removing all vertices of degree 1 (i.e., leaves) from $G$. We have that a graph $G$ is a caterpillar if its interior is a path.

\vs

We will concern ourselves with several ordered objects (functions, tuples, strings) which will be used somewhat interchangeably. A tuple $t \in \N^s$ is simply a function $t:[s]\to \N$. A binary string $x \in \{0,1\}^s$ is simply a function $x:[s] \to \{0,1\}$. For any function, $f$ we let $\dom(f)$ refer to its domain, the values $x$ where $f(x)$ is defined.

\vs

For a function/tuple/string $f \in \N^s$, we let $f'$ denote it's reversal (i.e. $f'(i) = f(s-i+1)$). We say that $f \sim g$ if $f = g$ or $f = g'$; it is clear that this is an equivalence relation. We also define $|f|_1$ to be $\sum_{i \in [s]} |f(i)|$ and $|f|_\infty$ to be $\max_{i \in [s]}\{|f(i)|\}$.

\vs

Given a set $S \subset [s]$ and $f \in \N^s$, we define $f|_S:[|S|]\to \N;i\mapsto f(S_i)$ (where $S_1 \le \dots \le S_{|S|}$). If $S \not \subset [s]$ and $f \in \N^s$, then $f|_S$ is not defined or is the empty function. We then define $D_S(f)$, the $S$-deck of $f$, to be the multiset of functions $g:[|S|] \to \N$ where the multiplicity of $g$ is $\sum_{i \in \Bbb{Z}} I(g = f|_{i+S})+I(g = f'|_{i+S})$. 

\hide{In other words, $D_S$ is the multiset of all $f|_T$, where $T$ ranges over all sets $T$ which can be obtained from $S$ by applying a translation and (possibly) a reflection. The reason for this unoriented and affine equivalence class is that given a set of vertices in a path graph, and the distance between them, this is the ``best'' equivalence class we can hope to achieve.\edit{revisit this explanation}}

\vs

For two functions $f,g \in \N^s$, we write \[F(f,g) = \prod_{t=1}^s \binom{f(i)}{g(i)}\]to denote the ``oriented embeddings'' of $g$ into $f$.

\vs

For a multiset $M$ of functions $f \in \N^s$, we write $d(M)$ to denote the subdeck (or moments) of $M$, which is the multiset $M^*$ where each $g \in \N^s$ has multiplicity $\sum_{f \in M} F(f,g)$. We let $d_\ell(M)$ (the $\ell$-moment of $M$) denote the restriction of $d(M)$ to functions $g$ where $|g|_1 \le \ell$.

\vs

For a set $S$, we define the \textit{width} of $S$ to be $\Delta(S) = \max(S)-\min(S)+1$. We say $S \prec T$ if $\Delta(S) < \Delta(T)$, and $|S| \le |T|$.

\vs

For a multiset $M$, we will sometimes write $\#(x \in M)$ to denote the multiplicity of the element $x$ within $M$.

\subsection{Plan of attack}\label{summary}

It is desirable to motivate these definitions, and explain how they will be used in the paper. It is quite plausible that our methods can be extended to further classes of graphs (especially trees\hide{ as we will discuss in Section~\ref{generalize}}), thus our outline is written a bit more generally to reflect this.

\vs 

Consider a graph $G=(V,E)$ equipped with the coloring $c:V\to \N ; v\mapsto \deg(v)$. We define the ``labelled pruning'' of $G$ to be the labelled graph $G[\Inner(G)]$ equipped with the coloring $c|_{\Inner(G)}$. Letting $\varphi$ be the map which sends $G$ to its labelled pruning, we see that $\varphi$ is an injection (unless $G$ has a connected component with $<3$ vertices, which is a case we need not consider\footnote{Furthermore, since the $2$-deck of a graph $G$ tells us how many edges and vertices $G$ has, one can see that if $\varphi(G_1) = \varphi(G_2)$ and $\mathcal{D}_2(G_1) = \mathcal{D}_2(G_2)$ then $G_1 \cong G_2$. Hence, if we can reconstruct $\varphi(G)$ from $\mathcal{D}_\ell(G)$ for $\ell \ge 2$, then we can reconstruct $G$ from its $\ell$-deck.}). The methods of this paper boil down to reconstructing $\varphi(G)$ when $G$ is a caterpillar. 

\vs

By definition, since $G$ is a caterpillar, we will have that $\varphi(G)$ is a labelled path graph. Hence, we can instead think of $\varphi(G)$ as being a function $f\in \N^s/\sim$. Thus, at a high-level our approach involves three steps:
\begin{enumerate}
    \item confirm $G$ is a caterpillar graph, and recognize its diameter,
    \item use $\mathcal{D}_\ell(G)$ to reconstruct $D_S(\varphi(G))$ for some appropriate sets $S$,
    \item use the $S$-decks $D_S(\varphi(G))$ to reconstruct $\varphi(G)$. 
\end{enumerate}Because caterpillars are quite structured, step (1) becomes trivial thanks to results from recent literature \cite{deg,tree} (this is covered at the start of Section~\ref{reconstructing caterpillars}). 

\vs

This leaves steps (2) and (3). The key question becomes ``what are the appropriate sets $S$ in step (2)''? As we make the sets $S$ larger (with respect to inclusion as well as width), the more information the $S$-decks convey --- this results in a trade-off where step (2) becomes harder to do efficiently yet step (3) becomes easier to handle. To give some context for this trade-off, we consider the extremes. 
\hide{

On one hand, we need to be able to use these sets to complete step (3). This encourages us to consider larger sets $S$. At a minimum, due to constructions discussed in Section 1, it will be necessary to include sets $S$ with width at least half the diameter of $G$.

However, we must know a way to reconstruct $D_S(\varphi(G))$ from $\mathcal{D}_\ell(G)$ without taking $\ell$ to be too large. This of course becomes more difficult as we make the sets $S$ larger/wider.}

\vs

First, what if we only use sets $S$ which are intervals (so that they are as large as possible given their width)? This should make step (3) as easy as possible, in the sense that if $\Delta(S) = w$ then the $[w]$-deck will carry at least much information as the $S$-deck and thus be at least as useful in step (3). Nevertheless, constructions mentioned in Section~\ref{intro} show that we will need to use some sets $S$ with width at least roughly half the diameter $G$ (or add some extra steps to our reconstruction process).

\vs

Now, the task of reconstructing $D_{[w]}(\varphi(G))$ had already been implicitly considered in \cite{deg} (in their language, this is essentially equivalent to reconstructing the 1-ball deck of a $w$-vertex path within $G$). However their method of reconstruction (\cite[Lemma~14]{deg}) used inclusion-exclusion, which for some caterpillars required decks that were too large for purposes. In particular, this was poor at handling caterpillars whose leaves were highly concentrated in a certain area (if there was a $w$-vertex path in $G$ which was neighboring $\ell-w$ leaves in $G$, then we were expected to look at the $\ell$-deck; thus if $G$ has diameter $4n/5$ and all its leaves are in the middle third $G$, we seemingly would need the $3n/5$-deck to reconstruct $D_{[2n/5]}(\varphi(G))$, and as a more exteme example if $G$ had a vertex of degree $9n/10$ this approach would require the $9n/10$-deck). We were unable to find a way to more efficiently reconstruct these $S$-decks and were forced to consider other options.

\vs

At the other extreme, we could try only using sets $S$ with 1 element. Here, it will be impossible to do step (3), because any set with one element has width 1 (which is less than half the diameter of $G$ for almost every caterpillar $G$). However, step (2) can be done much more efficiently. Indeed, reconstructing $D_{\{1\}}(\varphi(G))$ follows from reconstructing the degree sequence of $G$, which by \cite[Theorem~7]{deg} can be done by considering the $O(\sqrt{n\log(n)})$-deck of $G$. It is noteworthy that this bound is sublinear, this was achieved by an algebraic approach which bounded the number of shared moments between distinct sequences.

\vs 

Now let's discuss the sets $S$ used in this paper. We were able to show that step (3) can be done when provided the $S$-decks for all sets $S$ with up to 3 elements and width up to $\lfloor k/2+1\rfloor$ (where $k$ is the diameter of $G$). This is Lemma~\ref{func deck}, which we prove in Section~\ref{binary strings}. 

\vs

At the same time, we also were able to efficiently do step (2) for these sets. This was done by extending the algebraic techniques used in \cite[Theorem~7]{deg} in two substeps:
\begin{enumerate}
    \item[(2')] use graph enumerating arguments to show that we can reconstruct $\ell_0$-th moment of $D_S(\varphi(G))$ using $\mathcal{D}_{\Delta(S)+\ell_0}(G)$,
    \item[(2'')] show that when $|S|\le 3$, that we can reconstruct the $S$-deck by calculating $o(n)$ of its moments. 
\end{enumerate}

\vs

\section{Shared moments in multiple dimensions}\label{mom}

We say two sequences of tuples, $\alpha, \beta \in (\N^s)^m$, are related to each other by a permutation, if there is $\pi \in S_m$ such that $\alpha_i = \beta_{\pi(i)}$ for all $i \in [m]$.

\vs

The main result of this section is the following. Essentially, it says that to reconstruct a multiset $M\subset \N^s$ which we know is ``$(n,s)$-bounded'' in some sense, then it suffices to know the $\gamma_s(n)$-th moment of $M$. Or in other words, it says $d_{\gamma_s(n)}(\cdot)$ is injective on the set of ``$(n,s)$-bounded'' multisets.

\begin{lem}\label{alg} For each positive integer $s$, there exists a function $\gamma_s:\N\to\N$ of sublinear growth (in particular, we may take $\gamma_s(n) \le (1+o_s(1))\sqrt{n}\log n$) so that the following holds. 

\vs

Let $\alpha,\beta \in (\N^s)^m$ be two sequences that are not related to each other by a permutation. Suppose $\sum_{t=1}^m |\alpha_t|_1 \le n$, $\sum_{t=1}^m |\beta_t|_1 \le n$, and $\sum_{t=1}^m F(\alpha_t,\vec{i})= \sum_{t=1}^m F(\beta_t,\vec{i}) $ for all $\vec{i} \in \{0,\dots,\ell\}^s$. Then $\ell < \gamma_s(n)$.
\end{lem}Lemma~\ref{alg} immediately implies the following corollary, which is the only result from Section~\ref{mom} that will be used in future sections. We will use this to reconstruct $S$-decks (of functions associated with caterpillars) from their moments.
\begin{cor}
Take $\gamma := \max\{\gamma_1,\gamma_2,\gamma_3\}$, where $\gamma_1,\gamma_2,\gamma_3$ are the functions given by Lemma~\ref{alg}. We have that $\gamma:\N\to\N$ is sublinear (i.e., $\gamma(n) = o(n)$).
\end{cor}
\vspace{1.5mm}

To prove Lemma~\ref{alg}, we need a technical lemma that requires some additional definitions.

\vs

For two $\vec{i},\vec{j} \in \N^s$, we say that $\vec{i} \equiv_p \vec{j} $ if for all $k \in [s]$ we have $i_k \equiv j_k \mod{p}$. For a finitely supported function $a: \N^s\to \Bbb{Z}$, and $\vec{j} \in \N^s$, we define $\eta_{\vec{j},p}(a) = \sum_{\vec{x} \equiv_p \vec{j}} a(\vec{x})$.

\vspace{1mm}

\begin{lem} \label{num} Fix any $s$ and any absolute constant $C$. Suppose that $a:\N^s \to \Bbb{Z}$ is such that $|a|_1 + \sum_{\vec{x} \in \supp(a)} |\vec{x}|_1 \le Cn$. Suppose also that for every $\vec{j} \in \N^s$, and all primes $p \le (1+o(1))\sqrt{n }\log(n)$, we have $\eta_{\vec{j},p}(a) \equiv 0 \mod p$. Then $a$ is identically zero.
\end{lem}
\noindent This is our main number-theoretic result. Since moments and polynomials are closely connected, we will be able to reduce Lemma~\ref{alg} to Lemma~\ref{num} by the following argument, which is adapted from work in a previous paper by Scott \cite[Proof of Lemma~1]{scott} (in their notation $\eta_{\cdot,\cdot}(\cdot)$ was denoted as $n_{\cdot,\cdot}(\cdot)$).

\vs

\begin{proof}[Proof of Lemma~\ref*{alg}, assuming Lemma~\ref*{num}]

Consider $\alpha,\beta \in (\N^s)^m$ and $\ell$ such that
\[\sum_{t=1}^m F(\alpha_t,\vec{i})=\sum_{t=1}^m F(\beta_t,\vec{i})\] for all $\vec{i}\in\{0,\dots,\ell\}^s$. Moreover take $n$ so that $\sum_{t=1}^m|\alpha_t|_1,\sum_{t=1}^m|\beta_t|_1 \le n$.

\vs

Without loss of generality, we may assume $\alpha,\beta$ are disjoint, as removing like terms will not cause our assumptions about $\ell,n$ to become invalid. In particular, we shall assume $0^s\neq \alpha_t$ for any $t\in [m]$. This implies that $m \le \sum_{t=1}^m |\alpha_t|_1 \le n$ (hence $m$ is bounded in terms of $n$).

\vs

Let $a = a_{\alpha,\beta}:\N^s \to \Bbb{Z}; \vec{x} \mapsto |\{t: \alpha_t = \vec{x}\}| - |\{t:\beta_t = \vec{x}\}|$. We have that $|a|_1 + \sum_{\vec{x} \in \supp(a)} |\vec{x}|_1 \le  2m+ \sum_{t=1}^m |\alpha_t|_1+|\beta_t|_1\le 4 n$. Clearly, $\alpha,\beta$ are related to each other by a permutation if and only if $a$ is indentically zero.

\vs

Now, for each $i \in \N$, we have that $p_i:\N\to \N;x\mapsto \binom{x}{i}$ is an $i$-degree polynomial. It is clear that $p_0,\dots,p_\ell$ form a basis over 1-variable polynomials of degree at most $\ell$. Now, for $\vec{x},\vec{i} \in \N^s$, we have that $F(\vec{x},\vec{i}) = \prod_{k \in [s]} p_{i_k}(x_k)$. Through induction, we get that $\{F(\cdot,\vec{i}): \vec{i} \in \{0,\dots,\ell\}^s\}$ spans all $s$-variable polynomials where the degree in each variable is at most $\ell$ (basically we can fix the first variable, apply the inductive hypothesis on the other $s-1$-variables, then fix everything but the first variable and apply the 1-variable case).

\vs

Now, assume that $\delta_{\vec{i}} := \sum_{t=1}^m F(\alpha_t,\vec{i}) - F(\beta_t,\vec{i})=0$ for all $\vec{i} \in \{0,\dots,\ell\}^s$. Without loss of generality, we may assume that $\alpha,\beta$ are disjoint, as like terms will cancel out. In particular, we shall assume $0^s \neq \alpha_t$ for any $t \in [m]$. This implies that $m \le \sum_{t=1}^m |\alpha_t|_1 \le n$ (hence $m$ is bounded in terms of $n$).

\vs

Let $a = a_{\alpha,\beta}:\N^s \to \Bbb{Z}; \vec{x} \mapsto |\{t: \alpha_t = \vec{x}\}| - |\{t:\beta_t = \vec{x}\}|$. We have that $|a|_1 + \sum_{\vec{x} \in \supp(a)} |\vec{x}|_1 \le  2m+ \sum_{t=1}^m |\alpha_t|_1+|\beta_t|_1\le 4 n$. We also must have that $\sum_{\vec{x} \in \N^s} a(\vec{x}) \prod_{k \in [s]} x_k^{i_k} = 0$ for all $\vec{i} \in \{0,\dots,\ell\}^s$, as it is a linear combination of $\{\delta_{\vec{j}}:\vec{j} \in \{0,\dots,\ell\}^s\}$ (this follows from our fourth paragraph). 

\vs

For any prime $p$ and integer $N$, we have that $N^{p-1} \equiv 1 \mod{p}$ unless $N \equiv 0 \mod{p}$. Exploiting this fact, we can deduce that the moments of $a$ being zero for all $\vec{i} \in \{0,\dots,\ell\}^s$ implies that for all primes $p \le \ell$ and all $\vec{j} \in \N^s$, we have that $\eta_{\vec{j},p}(a) \equiv 0 \mod p$. We defer the details to Appendix~\ref{app}.

\vs

By Lemma~\ref{num}, it follows that either $\ell < (1+o(1))\sqrt{n}\log(n)$ or $a$ is identically zero, the latter of which would be impossible as we assume $\alpha,\beta$ are not related by a permutation.
\end{proof}

We will now prove Lemma~\ref{num}. We remark that our argument can actually prove stronger results than what is stated in Lemma~\ref{num}, we discuss the details in Remark~\ref{fullstrength}.

\begin{proof}[Proof of Lemma~\ref*{num}]
Let $k = n^{1/2}$, and let $P$ denote the set of primes in the interval $[k\log\log n,k\log n]$. By the Prime Number Theorem, we have that $|P| = (2-o(1))k$.

\vs

For a prime $p$, we say that $a$ is $p$-good if $\eta_{\vec{j},p}(a) \equiv 0 \mod{p}$ for all $\vec{j} \in \N^s$. We suppose that $a:\N^s\to \Bbb{Z}$ is $p$-good for all primes up to $k\log(n)$.

\vs

We shall proceed to prove the contrapositive, assuming $\Supp(a) \neq \{\}$ and deducing that \[|a|_1 + \sum_{\vec{x} \in \supp(a)} |\vec{x}|_1 = \omega(n).\]

\vs

First, we suppose there are distinct $\vec{x},\vec{y} \in \supp(a)$ and distinct $p_1,p_2 \in P$ such that $\vec{x} \equiv_{p_1}\vec{y}$ and $\vec{x} \equiv_{p_2}\vec{y}$. Since $\vec{x},\vec{y}$ are distinct, there is $i\in [s]$ such that $x_i-y_i \neq 0$. Furthermore by assumption we'll have that $p_1p_2 \mid x_i-y_i$. Since $p_1,p_2\in P \subset [k\log\log n,k\log n]$, we get that \[n(\log\log n)^2 \le |x_i-y_i| \le |\vec{x}|_1 +|\vec{y}|_1 \le |a|_1 + \sum_{\vec{x} \in \supp(a)} |\vec{x}|_1.\]Thus, this scenario implies that $|a|_1 + \sum_{\vec{x} \in \supp(a)} |\vec{x}|_1 = \omega(n)$, as desired. We now proceed, assuming that for distinct $p_1,p_2\in P$, $\vec{x}\in \supp(a)$ and $\vec{y}_1,\vec{y}_2\in \supp(a)\setminus \{\vec{x}\}$ that
\[\vec{x}\equiv_{p_1}\vec{y}_1, \vec{x}\equiv_{p_2} \vec{y}_2\implies \vec{y}_1\neq \vec{y}_2.\tag{$*$}\label{disjoint}\]

\vs

Since $\supp(a)$ is non-empty, we may choose $\vec{x}_* \in \supp(a)$ such that $|\vec{x}_*|_\infty$ is minimal (such a choice might not be unique). We observe
\begin{align*}
    |a|_1 &\ge |a(\vec{x}_*)| \tag{$\dagger$}\label{valbound}\\
    \sum_{\vec{x}\in \supp(a)}|\vec{x}|_1 &\ge \sum_{p\in P}\sum_{\substack{\vec{y}\in \supp(a)\setminus \{\vec{x}_*\}\\:\vec{y} \equiv_p \vec{x}_*}} |\vec{y}|_\infty  \tag{$\ddagger$}\label{corbound}.
\end{align*}(Eq.~\ref{valbound} is immediate, while Eq.~\ref{corbound} uses \ref{disjoint} to deduce that no $\vec{y}$ is ``double-counted''.) We shall show that the RHS of either Eq.~\ref{valbound} or Eq.~\ref{corbound} will be $\omega(n)$ (which allows us to deduce the sum of their LHS's is $\omega(n)$, as desired). 

\vs 

To this end, we define $P_*= \{p\in P: p\mid |a(\vec{x}_*)|\} $. We shall prove
\begin{align}
    |a(\vec{x}_*)| &\ge (k\log\log n)^{|P_*|}\label{valbnd}\\
    \sum_{p\in P}\sum_{\substack{\vec{y}\in \supp(a)\setminus \{\vec{x}_*\}\\:\vec{y} \equiv_p \vec{x}_*}} |\vec{y}|_\infty &\ge |P\setminus P_*|(k\log\log n).\label{corbnd}
\end{align}
Recalling from earlier that $|P| = (2-o(1))k$, we may apply pigeonhole principle to see $\max\{|P_*|,|P\setminus P_*|\}\ge (1-o(1))k$. So, if $|P^*| \ge (1-o(1))k$, then Eq.~\ref{valbnd} implies that the RHS of Eq.~\ref{valbound} is\footnote{In fact, here the RHS will be exponentially large. We elaborate upon the ``full strength'' of our proof in Remark~\ref{fullstrength}.} $\omega(n)$. Similarly, if $|P\setminus P^*| \ge (1-o(1))k$, then Eq.~\ref{corbnd} shows that the RHS of Eq.~\ref{corbound} is at least $(1-o(1))n\log\log n = \omega(n)$. In either case we are done, thus we are left to justify Equations~\ref{valbnd} and ~\ref{corbnd}.

\vs

Deriving Eq.~\ref{valbnd} is straight-forward. Since $\vec{x}\in \supp(a)$, $|a(\vec{x}_*)|$ is a positive integer. Hence, as $P_*$ is a set of primes which divide $|a(\vec{x}_*)|$, we get\footnote{In the equation below, we make a minor abuse of notation. If $P^*$ is the empty set, we consider $\min\{p \in P^*\}^{|P^*|}$ to equal $1$.}\[|a(\vec{x}_*)|\ge \prod_{p\in P_*}p \ge \min\{p\in P_*\}^{|P_*|}\ge  (k\log\log n)^{|P_*|}.\](For the last inequality, we recall $P_*\subset P\subset [k\log\log n,k\log n]$.) Hence, we have obtained Eq.~\ref{valbnd}. 

\vs

To verify Eq.~\ref{corbnd}, we first define $P' = P\setminus P_*$. For each $p\in P'$, we claim there exists some $\vec{y} \in \supp(a) \setminus \{\vec{x}_*\}$ such that $\vec{y} \equiv_p \vec{x}_*$. Indeed, since $p\in P'$, $p\nmid |a(\vec{x}_*)|$ thus if there was no such $\vec{y}$, we'd have that
\[\eta_{\vec{x}_*,p}(a) = a(\vec{x}_*) \not \equiv 0\mod{p},\]contradicting the assumption that $a$ is $p$-good. Hence, there must exist some such $\vec{y}\in \supp(a)\setminus \{\vec{x}_*\}$. This allows us to define a function $\phi:P'\to \supp(a)\setminus \{\vec{x}_*\}$ such that $\vec{x}_* \equiv_p \phi(p)$ for each $p \in P'$.

\vs

It follows that
\[\sum_{p\in P}\sum_{\substack{\vec{y}\in \supp(a)\setminus \{\vec{x}_*\}\\:\vec{y} \equiv_p \vec{x}_*}} |\vec{y}|_\infty \ge \sum_{p \in P'} |\phi(p)|_\infty.\]
To obtain Eq.~\ref{corbnd}, we shall show that each summand on the RHS above is at least $k\log\log n$. This will follow from our choice of $\vec{x}_*$. 

\vs

There are two cases. If $|\vec{x}_*|_\infty \ge k\log\log n$, then recalling $\vec{x}_*$ minimizes the infinity norm, we have $|\vec{y}|_\infty \ge k\log\log n$ for all $\vec{y} \in \supp(a)$ (meaning we are done). Otherwise, $|\vec{x}_*|_\infty < k\log\log n$. Suppose for sake of contradiction there was $p\in P'$ such that $|\phi(p)|_\infty<k\log\log n$. 

\vs

It should then follow that $|\phi(p)-\vec{x}_*|_\infty <k\log\log n\le p$, because both vectors do not have negative entries and hence are in $[0,k\log\log n-1]^s$. Furthermore, since $\vec{x}_* \equiv_p \phi(p)$, we should have $ \phi(p)-\vec{x}_* \equiv_p 0^s$. By our last two sentences, we should get that  $\phi(p)-\vec{x}=0^s$. However, this contradicts $\phi(p) \in \supp(a)\setminus \{\vec{x}_*\}$. And so, we must have that $|\phi(p)|_\infty \ge k\log\log n$ as desired.

\end{proof}

\begin{rmk}\label{fullstrength}Our proof of Lemma~\ref{num} can easily give a stronger result. In particular, it would still work if we only required that the RHS of Eq.~\ref{valbound} had at most $k$ divisors in $P$ (because conditioned on this not happening we have that $|P'| \ge (1-o(1))k$). Thus, we can show that either $|a|_\infty > (k\log\log n)^k$ or $\sum_{\vec{x} \in \supp(a)}|\vec{x}|_\infty \ge (1-o(1))n\log\log n$. This will not be necessary for our purposes, however.
\end{rmk}

\section{Reconstructing Caterpillars} \label{reconstructing caterpillars}

We associate each $k+1$ diameter caterpillar $G$ with a function $f:[k] \to \N$, where $f(i) = \deg(v_{i+1})-2$ where $v_1,\dots, v_{k+2}$ is some arbitrarily determined maximum path in $G$. This defines an injection $\phi:G\mapsto f$, where $G\neq H \implies \phi(G) \not \sim \phi(H)$. Given a function $f:[k] \to \N$, we note that it is trivial to reconstruct $\phi^{-1}(f)$.

\vs

In this section, we shall ultimately show that 
\begin{lem}\label{cat2func} Suppose we know $G$ is a caterpillar graph with diameter $k+1$ and $n$ vertices. Let $\ell \ge w + 3\gamma(\rho(n-k))\hide{(3+o_{n-k}(1))\sqrt{n-k}\log(n-k)}$(where $\rho(x) = c_1x +c_2$ is an affine linear function). Given $\mathcal{D}_\ell(G)$, and the degree sequence of $G$, we can reconstruct $D_S(\phi(G))$ for all $S$ where $\Delta(S) \le w$, and $|S| \le 3$.
\end{lem}
\begin{rmk}
Inspecting our proofs, one may confirm that we can in fact reconstruct the desired decks $D_S(\phi(G))$ by looking at $\mathcal{D}_\ell(G)$ restricted to the cards with diameter at most $w+1$ (along with the rest of the information assumed to be provided).
\end{rmk}

\vs

And in Section~\ref{binary strings}, we prove
\begin{lem}\label{func deck}Given two functions $f,g:[k]\to \N$, either $f \sim g$, or there is some $S$ with $|S|\le 3$ and width at most $\lfloor k/2-1\rfloor  +1$ such that $D_S(f) \neq D_S(g)$.
\end{lem}

\vs

Assuming these two lemmas, we can prove our main result.

\begin{proof}[Proof of Theorem~\ref{main}, assuming Lemma~\ref{cat2func} and Lemma~\ref{func deck}]

By looking at the $1$-deck of $G$, we can construct its number of vertices, $n$.

\vs

Let $\ell_1(n)$ be such that we can recognize whether any graph $G$ on $n$ vertices is a caterpillar graph by looking at its $\ell_1(n)$-deck. Let $\ell_2(n)$ be such that we can reconstruct the degree sequence of any graph $G$ with $n$ vertices by looking at its $\ell_2(n)$-deck; as we will note below if we know the degree sequence of $G$ and that $G$ is a caterpillar, we can reconstruct diameter of $G$.\hide{Let $\ell_2(n)$ be such that we can reconstruct the diameter any of caterpillar graph $G$ with $n$ vertices by looking at its $\ell_2(n)$-deck.} We let $\ell_3(n) = \max_{k \in [n-2]} \{k/2+1+ 3\gamma(\rho(n-k))\}$.

\vs

If $G$ is a caterpillar graph with $n\ge 3$ vertices, then its diameter is $k+2$ for some $k \in [n-2]$. Hence, by Lemma~\ref{cat2func} and Lemma~\ref{func deck}, taking $\ell(n) = \max\{1,\ell_1(n),\ell_2(n),\ell_3(n)\}$ we can reconstruct $\phi(G)$, and hence $G$, by looking at $\mathcal{D}_{\ell(n)}(G)$. We make a few observations to show that $\ell(n) = n/2+O(1)$, giving our main result.

\vs

We have that $G$ is a caterpillar if and only if it is a tree and does not contain a certain 7 vertex subgraph ($K_{1,3}$ with all edges subdivided). In a recent paper by Kostochka, Nahvi, West, and Zirlin, it was proved that we can recognize whether a graph is a tree by using the $(\lfloor n/2 \rfloor +1)$-deck \cite{tree}. Meanwhile, by Kelly's lemma (Remark~\ref{kelly}), if $\ell \ge V(H)$, then we construct $n_H(G)$, the number of subgraphs of $H$ in $G$, by looking at $\mathcal{D}_\ell(G)$. Hence, we see that $\ell_1(n) = \max\{7,\lfloor n/2\rfloor +1\} = n/2 + O(1)$.

\vs

Meanwhile, by Remark~\ref{dseq}, we can reconstruct the degree sequence of any $n$-vertex graph using $O(\sqrt{n\log(n)})$ vertices. So, we get that $\ell_2(n) = O(\sqrt{n\log(n)})= n/2 +O(1)$. Furthermore, when $G$ is a caterpillar, its diameter is $n-|\{v : \deg(v) = 1\}| +2$, so we can reconstruct this quantity by looking at the $\max\{\ell_1(n),\ell_2(n)\}$-deck.

\vs

Lastly, we note that $k/2+3\gamma(\rho(n-k))+2-n/2 = 3\gamma(\rho(n-k))+2-(n-k)/2$. Since $\gamma$ is sublinear, and $\rho$ is affine, we have that there is some $C$ such that $3\gamma(\rho(n-k))+2-(n-k)/2 < C$ for all values $n-k \in \N$. And so $\ell_3(n) \le n/2+C$.
\end{proof}

\vs

\subsection{Some details and exposition}

We will now go over an outline of our proofs.

\hide{We associate a $k$ diameter caterpillar $G$ with a function $f:[k] \to [0,n]$, where $f(i) = \deg(v_{i+1})-2$ where $v_1,\dots, v_{k}$ is some arbitrarily determined maximum path in $G$. This defines a injection $\phi:G\mapsto f$. Given a function $f:[k] \to [-1,\infty)$, we note that it is trivial to reconstruct $\phi^{-1}(f)$.}

\vs

\hide{For a function $f:[k] \to [2,\infty)$, and a set $S\subset [k]$, we let $D_S(f)$ be the multiset of functions $g:[|S|] \to [2,\infty)$ where the multiplicity of $g$ is $\sum_{i \in \Bbb{Z}} I(g = f_S)+I(g = f'_S)$.}

As was proved in Lemma~\ref{alg}, to reconstruct a multiset $M$, it suffices to take enough moments (i.e. calculate $d_\ell(M)$ for large enough $\ell$). Thus, a reasonable idea would be to try and calculate the moments of the deck $D_S(\phi(G))$ for each $S$ mentioned in Lemma~\ref{cat2func}. We will essentially be doing this, except we first apply an injective transformation to our multiset, and then take moments.

\vs

We define a translation $T$, which will act on each set $\Z^s$ (it will be explained in a few paragraphs why we will be working over $\Z^s$ rather than $\N^s$). For a function $g \in \Z^s$, let $T(g)$ be the function obtained by incrementing the values of $g(1)$ and $g(s)$ by 1 (here if $s =1$, then $T(g)(1) = g(1)+2$). For a multiset of functions, $M$, we let $T(M)$ be multiset of $T(g)$ for $g \in M$. It is clear that $T$ is an injective map from multisets to multisets.

\vs

The motivation behind $T$ is this. Suppose $G$ is caterpillar, with $v_0,v_1,\dots,v_k,v_{k+1}$ being its identified path. We have that $T(\phi(G))(i)$ counts the number of leaves of $v_i$. This leaf counting function will be much more natural to consider when working with subgraph counts, due to counting identities which will be shown later (cf. Eq.~\ref{basicidentity}).

\vs

We call a function $g \in \N^s$ \textit{degenerate} if $g(1) = 0$ and/or $g(s) = 0$, and otherwise call $g$ non-degenerate. As we shall promptly explain, calculating the moments of degenerate $g$ versus non-degenerate $g$ are rather different tasks, and will be proven in two parts. Because of this casework, for a multiset of functions, $M$, we let $M_{degen}$ be the multiset which only contains $g\in M$ which are degenerate, and define $M_{non-degen}$ to be $M \setminus M_{degen}$.

\vs

When $G,H$ are caterpillar graphs with the same diameter, we have that $T(\phi(H))$ is non-degenerate, and \[F(T(\phi(G)),T(\phi(H)))+F(T(\phi(G)),T(\phi(H)'))= n_H(G)(1 + 1_E) \tag{$\S$}\label{basicidentity}\] where $E$ is the event that $\phi(H)$ is symmetric. In Proposition~\ref{nondegen}, we will show that whenever $g$ is non-degenerate, then the (multiplicity $g$ in $d(T(D_S(\phi(G))))$/moment of $g$ in $T(D_S(\phi(G)))$) will similarly count $n_H(G)$ for some caterpillar $H$, which can be calculated by looking at $\mathcal{D}_{|V(H)|}(G)$.

\vs 

Meanwhile, when $g$ is degenerate, the moment of $g$ in $T(D_S(\phi(G)))$ cannot be as nicely expressed with subgraph counts. Here, the moment roughly counts the number of times a certain $H$ appears as a subgraph in $G$, with certain leaves of $H$ lying in the interior of $G$. Having leaves be in the interior of $G$ is a rather unwieldy boundary condition; to circumvent this, we actually work with a slightly different function than $\phi(G)$.

\vs 

Given a function $f:[k] \to \N$, we define the auxillary function $h = \aux(f)$ so that $h(0) = -1 = h(k+1)$, and $h|_{[k]} = f$. In Proposition~\ref{degen}, we show that for degenerate $g$ that we can calculate its moment in $T(D_S(\aux(\phi(G))))$ through induction. Meanwhile, for non-degenerate $g$, the moments will not change (essentially\hide{ since in each deck of the auxillary function, we only get two new cards, which both have a zero at one of their ends, meanwhile $g$ being non-degenerate has a value of at least 1 at each of its ends} $\aux(\phi(G))(0) = 0< g(1)$, and so when you compute $F$ for additional cards in the deck, we will have a factor of $\binom{0}{g(1)} = 0$). Thus everything nicely works out here.

\vs

\subsection{Proofs}

We will now use the $\ell$-deck of $G$ to reconstruct some information about $\phi(G)$. For the following sections, let $G$ be a caterpillar graph, $f = \phi(G)$, $h = \aux(f)$

\begin{prp} \label{nondegen} For $S\subset [k],|S| \ge 2, \ell_0 \ge 0$. Suppose $\ell \ge \ell_0+\Delta(S)$. Given $\mathcal{D}_{\ell}(G)$, we can reconstruct $d_{\ell_0}(T(D_S(h)))_{non-degen}$.
\begin{proof} WLOG, assume $\min(S) = 1$.

\vs

Consider some $g:[|S|] \to \N$ where $|g|_1 \le \ell_0$, where $\supp(g) \supset \{1,|S|\}$. We shall calculate the multiplicity of $g$ in $d(T(D_S(h)))$. Iterating over all such $g$, we will have reconstructed $d_{\ell_0}(T(D_S(h)))_{non-degen}$.

\vs

As $g$ is non-degenerate, there should be a caterpillar $H_g$ where $T(\phi(H_g))\sim g$. When we subdivide edges of the inner path of a caterpillar $H$, we do not change the fact that $H$ is a caterpillar, and introduce a new vertex with zero leaves. So, by subdividing $H_g$, we can get a caterpillar $H^*$ so that $T(\phi(H^*))|_S \sim g$ and $\supp(T(\phi(H^*))) \subset S$.  

\vs

As alluded to earlier, we claim that the multiplicity of $g$ in $T(D_S(h))$ will be
\[n_{H^*}(G)(1+1_E)\]
where $E$ is the event that $\phi(H^*) = \phi(H^*)'$. By Kelly's lemma (Remark~\ref{kelly}), we have that $n_{H^*}(G)$ can be reconstructed by $|V(H^*)|$-deck. Now, $H^*$ has $|g|_1 \le \ell_0$ leaves, and $\Delta(S)$ non-leaf vertices, thus $|V(H^*)| \le \ell_0+\Delta(S)\le \ell$ meaning we can reconstruct the multiplicity as desired.

\vs

We shall now verify the proposed formula for the multiplicity of $g$. Now, we have that
\begin{align*}\#(g\in d(T(D_S(h)))) &= \sum_{i \in \Z} F(T(h|_{i+S}),g) +  \sum_{i \in \Z} F(T(h'|_{i+S}),g)\\
&= \sum_{i \in \Z} F(T(f|_{i+S}),g) +  \sum_{i \in \Z} F(T(f'|_{i+S}),g),\\\end{align*}
where in the last line, we use the fact that $T(D_S(h)) \setminus T(D_S(f)) = \{ T(h|_S), T(h'|_S)\}$ and $T(h|_S)(1) = 0 =T(h'|_S)(1)$, thus $F(T(h|_S),g) = 0 = F(T(h'|_S),g)$ as $g(1) > 0$.

\vs

Informally, we shall show that each summand corresponds to the number of copies of $H^*$ in $G$ where we specify that $\inner(H^*)$ embeds into $\inner(G)$ according to a specific map. The first sum handles all embeddings that are oriented in one direction, while the second sum handles all embeddings oriented in the opposition direction. Meanwhile, the appearance of $(1+1_E)$ is to handle the double counting which occurs when $H^*$ is symmetric, in which case direction of the embedding won't change the set of copies that can appear.

Let $v_1,\dots v_k$ be the inner path in $G$ and $u_1,\dots,u_{\Delta(S)}$ be the inner path of $H^*$ (meaning that $\deg(v_i) -2= \phi(G)(i) $, $\deg(u_i)-2 = \phi(H^*)(i)$). WLOG, we shall assume $\phi(H^*) = g$, otherwise as $\phi(H^*) \sim g$ we'd have $\phi(G) = g'$ in which case things follow in the same fashion.

We claim that for $i \in \Z$, \begin{align*}
    F(T(h|_{i+S}),g) &= |\{ U \subset V(G) :\textrm{there is an isomorphism } \psi: V(G[U]) \to V(H^*)\\ &\quad\quad \textrm{ s.t. }\psi(v_{i+j}) = u_j \textrm{ for each $j \in [\Delta(S)]$}\}|,\\
    F(T(h'|_{i+S}),g) &= |\{ U \subset V(G) :\textrm{there is an isomorphism } \psi: V(G[U]) \to V(H^*)\\ &\quad\quad \textrm{ s.t. }\psi(v_{k-(i-1+j)}) = u_j \textrm{ for each $j \in [\Delta(S)]$}\}|.\\
\end{align*}The above claim is straight-forward to verify. It suffices to consider the first claimed equality, as the second will follow by reversing $v_1,\dots,v_k$. We note that for $j\in [\Delta(S)]$\begin{align*}
    T(h|_{i+S})(j) &=\deg(v_{i+j})-|\{j-1,j+1\}\cap [\Delta(S)]| \\&= \deg(v_{i+j}) - |\{j'\in [\Delta(S)]: v_{i+j'}\textrm{ is adjacent to }v_{i+j}\}|\\
\end{align*}counts the ``leaves'' of $v_{i+j}$ in the subgraph $X$ of $G$ induced by $v_{i+1},\dots,v_{i+\Delta(S)}$ and their neighbors. Meanwhile, the number of leaves of $u_j$ in $H^*$ is $g(t)$ if $j$ is the $t$-th smallest element of $S$, and zero otherwise. Hence, as the LHS of the above counts the number of choices to embed the leaves of $H^*$ into the leaves of $X$, it equals  \begin{align*}
    \prod_{j=1}^{\Delta(S)} \binom{\textrm{number of leaves of }v_{i+j} \textrm{ in }X}{\textrm{number of leaves of }u_{j} \textrm{ in }H^*} 
    &= \prod_{j \in S} \binom{\textrm{number of leaves of }v_{i+j} \textrm{ in }X}{\textrm{number of leaves of }u_{j} \textrm{ in }H^*}\\ 
    &= F(T(H|_{i+S}),g)\\
\end{align*} (in the first equality we removed all the $j \not \in S$ since for such terms the bottom argument of the binomial will be zero whilst the top will be non-negative, thus such terms will equal one). Thus our claimed equality holds as desired.

With the claim established, we are almost done. It suffices to show that every copy of $H^*$ in $G$ must be of the two above forms, we leave this as an exercise to the reader.

\end{proof}
\end{prp}

\begin{prp}\label{degen} For $S \subset [k]$, $|S| \ge 2$. Given $D_{U}(h)$ for every $U \prec S$, we can reconstruct $d(T(D_S(h)))_{degen}$.
\begin{proof} Consider some $g:[|S|] \to \N$ where $\{1,|S|\} \setminus \supp(g) \neq \{\}$. WLOG, let's assume $1 \not \in \supp(g)$, the other case will follow in the same fashion. Let $E = [2,|S|]$ (in the other case, we would take $E = [|S|-1]$). 

\vs

We remark that for $C \in T(D_S(h))$, that $C(1) \ge 0$, thus $F(C,g) = F(C|_E,g|_E)$. Hence, writing $M = \{ C|_E: C \in T(D_S(h))\}$, we have that the multiplicity of $g$ in $d(T(D_S(h)))$ is equal to the multiplicity of $g|_E$ in $d(M)$. 

\vs

We will now reconstruct $M$, after which we easily get our result. Let $U = (S \setminus \{\min(S)\})\cup \{\min(S)+1\}$. It is clear that $\Delta(U) = \Delta(S)-1,|U| \le |S|$ and so $U \prec S$. Thus by assumption we know $T(D_U(h))$. We then get that for each card $c$, \begin{align*}\#(c \in M) &= \sum_{C \in T(D_S(h))} I(C|_E = c)\\
&=\sum_{i \in \Z}I(i+S \subset \dom(h)) (I(T(h|_{i+S})|_E = c) +I(T(h'|_{i+S})|_E = c))\\
&=\sum_{i \in \Z}I(i+S \subset \dom(h)) (I(T(h|_{i+U})|_E = c) +I(T(h'|_{i+U})|_E = c))\\
&= \sum_{C \in T(D_U(h))} I(C(1) > 0)I(C|_E = c)\end{align*}which can be reconstructed as we know $T(D_U(h))$.\footnote{In the fourth equality, we make use of the fact that $I(i+S \subset \dom(h)) \neq I(i+U \subset \dom(h)) $ occurs if and only if $i = -\min(S)$, which in turn is equivalent to $T(h|_{i+U})(1) = 0 \iff T(h'|_{i+U})(1) = 0$.}
\end{proof}
\end{prp}

We shall now make use of Section~\ref{mom}.
\begin{prp}\label{moments2slices} Suppose $|S| = s$, and $\ell \ge s\gamma(6s(n-k)+12s)$. Given $d_\ell(T(D_S(h))$, along with the number of vertices $n$ and the diameter $k+2$ of $G$, we can reconstruct $T(D_S(h))$ and $D_S(f)$.
\begin{proof} We shall apply Lemma~\ref{alg} to reconstruct $T(D_S(h))$. As $T$ is easily invertible, this allows us to reconstruct $D_S(h)$. Since $D_S(h) \subset D_S(f)$ and in fact $ g \in D_S(h) \setminus D_S(f)$ occurs if and only if $-1 \in \im(g)$, this allows us to reconstruct $D_S(f)$ as desired. 

\vs

Having been given $d_{\ell}(T(D_S(h)))$, we can compute 
\[\#(\vec{i} \in d(T(D_S(h)))) = \sum_{c \in T(D_S(h))} F(c,\vec{i})\]
for any $\vec{i} \in \{0,\dots, \lfloor \ell/s\rfloor\}^s$.

\vs

Now, recall that $|f|_1 = n-k$ is the number of leaves in $G$. Let $f^*:[k]\to \N$ be some other function where $|f^*|_1 = n-k$, and letting $h^* = \aux(f^*)$ we have that $T(D_S(h^*)) \neq T(D_S(h))$. We want to show that there exists $\vec{i} \in \{0,\dots, \lfloor \ell/s\rfloor \}^s$ such that the multiplicity of $\vec{i}$ in $d(T(D_S(h)))$ and $d(T(D_S(h^*)))$ differs.

\vs

Now, as $f,f^*$ both have a 1-norm of $n-k$, we get that $\sum_{c \in D_S(f)} |c|_1,\sum_{c \in D_S(f^*)} |c|_1$ are both at most $2s(n-k)$.\footnote{Indeed, for each index $j \in [k]$, there are at most $|S| = s$ translations of $S$ that contain $j$, and the factor of two comes from how we induce translates of $S$ in both $f$ and its reversal $f'$.} As $D_S(f),D_S(f^*)$ are multisets in $\N^s$, there are at most $2s(n-k)$ cards which are non-zero in either multiset (i.e., $\#(c \in D_S(f):c \neq 0^s),\#(c \in D_S(f^*): c\neq 0^s) \le 2s(n-k)$). Next, we let $t = \min\{\#(0^s \in D_S(f)),\#(0^s \in D_S(f^*))\}$ and $A$ be the multiset that just contains $t$ copies of $0^s$. We know that $D_S(f) =A\cup M_1, D_S(f^*) = A \cup N_1$ where $M_1,N_1$ are multisets having at most $2s(n-k)$ elements.

\vs 

Now, we have that $d$ is ``linear'', in the sense that for multisets $A,M,N\subset \N^s$, and $c,\vec{i} \in \N^s$, we will have that that
\[\#(\vec{i} \in d(M\cup\{c\})) - \#(\vec{i} \in d(N\cup\{c\})) = \#(\vec{i} \in d(M)) - \#(\vec{i} \in d(N))\]
\[\#(\vec{i} \in d(M\cup A)) - \#(\vec{i} \in d(N\cup A)) = \#(\vec{i} \in d(M)) - \#(\vec{i} \in d(N))\] (the second equality follows from the first by induction).

\vs

We observe that for any multiset $M$, that
\[\sum_{c \in T(M)} |c|_1 \le \sum_{c \in M} (|c|_1+2).\] Now, from the definition of $\aux$, it is clear that $|h|_1 = |f|_1+2, |h^*|= |f^*|+2$, which both are to equal $n-k+2$; this will imply that $\sum_{c \in D_S(h)}|c|_1, \sum_{c \in D_S(h^*)}|c|_1\le 2s(n-k+2)$. Also by the definition of $\aux$, it is clear that $D_S(h) = D_S(f) \cup M_2, D_S(h^*) = D_S(f^*)\cup N_2$, where $M_2,N_2$ are multisets each containing at most 4 cards. Writing $M = M_1 \cup M_2, N = N_1 \cup N_2$ which are multisets each having at most $2s(n-k)+4$ elements, we get that $D_S(h) = A \cup M, D_S(h^*) = A \cup N$.

\vs

Putting this together, we get
\[\sum_{c \in T(M)} |c|_1, \sum_{c \in T(N)} |c|_1 \le 2s(n-k+2) + 2(2s(n-k)+4) = 6s(n-k)+12s. \] By assumption, we chose $f^*$ so that $T(D_S(h^*)) \neq T(D_S(h))$, which implies that $T(N) \neq T(M)$. So, by Lemma~\ref{alg}, there must exist some vector $\vec{i} \in \{0,\dots, \gamma(6s(n-k) + 12s)\}^s$ such that \[\#(\vec{i} \in T(M))-\#(\vec{i} \in T(N)) = \sum_{c \in T(M)} F(c,\vec{i}) - \sum_{c \in T(N)} F(c,\vec{i}) \neq 0.\]By the ``linearity'' of $d$, we'll have that $\#(\vec{i} \in d(T(D_S(h)))) \neq \#(\vec{i} \in d(T(D_S(h^*)))$.

\vs

Meanwhile, we will have that $\lfloor \ell/s\rfloor \ge \gamma(6s(n-k)+12s)$ by assumption. Thus $\vec{i} \in \{0,\dots,\lfloor \ell/s\rfloor \}^s$, and thus given $d_\ell(T(D_S(h)))$ we can recognize $T(D_S(h)) \neq T(D_S(h^*))$. This allows us to reconstruct $T(D_S(h))$ and consequently $D_S(f)$, as desired. 
\end{proof}
\end{prp}

\begin{rmk}\label{rho}The affine function $\rho$ which appears in Lemma~\ref{cat2func} will take the explicit form $\rho:x\mapsto 18x+36$, which arises from plugging $s=3$ into Proposition~\ref{moments2slices}.
\end{rmk}

We are now ready to reconstruct $D_S(f)$.

\begin{proof}[Proof of Lemma~\ref*{cat2func}] In the prompt, we are given the degree sequence of $G$. This allows us to reconstruct $D_{\{1\}}(f)$. Indeed, if $d_1\le \dots \le d_n$ are the degrees of $G$, we have that $D_{\{1\}}(f)$ is exactly the non-negative values of $\{d_i -2\}$. As all singletons are translates of one another, and $|S| > 1 \implies \Delta(S) > 1$, this handles the case where $\Delta(S) = 1$. We note also that if $S$ is a singleton, that $T(D_S(h)) = T(D_S(f)) \cup \{1\}\cup\{1\}$.

\vs

We will now use induction to show that for any other $S$ where $\Delta(S) \le w, |S| \le 3$, that we can construct $d_{3\gamma(\rho(n-k))}(T(D_S(h)))$, which by Proposition~\ref{moments2slices} implies we can reconstruct $D_S(f)$. We induct on $\Delta(S)$, having already handled the case of $\Delta(S) = 1$.

\vs

Fix $1<w'\le w$ and assume we know $T(D_S(h))$ for all $S$ with $|S|\le 3$ and $\Delta(S)<w'$. Consider any $S$ with $|S|\le 3$ and $\Delta(S)=w'$. Note that as $\Delta(S) =w'>1$, we in fact have $|S| \ge 2$. By Proposition~\ref{nondegen}, as $\ell \ge \Delta(S) + 3\gamma(\rho(n-k))$ and $|S|\ge 2$, and we can reconstruct $d_{3\gamma(\rho(n-k))}(T(D_S(h)))_{non-degen}$. By Proposition~\ref{degen} and the inductive hypothesis (along with the fact that $|S|\ge 2$) we can also reconstruct $d(T(D_S(h)))_{degen}$. Thus, we have reconstructed $d_{3\gamma(\rho(n-k))}(T(D_S(h)))$, and thus by Proposition~\ref{moments2slices} may reconstruct both of the decks $T(D_S(h))$ and $D_S(f)$.
\end{proof}

\section{Unoriented Binary Strings}\label{binary strings}

\hide{Let $x\in \{0,1\}^k$ be a binary string and $x'$ denote its reversal (so that $x'(i) = x(k-i+1)$ for all $i \in [k]$). 

For $S \subset [k]$, we say that the $S$-subword in $x$ is the word $x_{S} := x(S_1) |x(S_2)| \dots |x(S_{|S|})$ (here $S_1< \dots < S_{|S|}$). For a set $S$, we let $i+S = \{i+a: a \in S\}$.

We let $D_S(x)$ denote a multiset of strings with $|S|$ letters, $M$, where the multiplicity of $w \in \{0,1\}^{|S|}$ is $\sum_{i  \in \Bbb{Z}} I(w = x_{i+S}) + I(w = x'_{i+S})$.

For a family of sets $F$, we let $D_F(x)$ denote a map from each element $S \in F$ to $D_S(x)$. We then let $D_{s,\ell}$ denote $D_F(x)$ where $F = \{ \{0\}\subset S \subset [0,\ell]:|S| \le s\}$.}

We let $\mathscr{D}_{s,\ell}(x)$ denote the map from each set $S$ where $\Delta(S) \le \ell,|S| \le s$ to $D_S(x)$. 

\vs

We shall prove 

\begin{lem} \label{reconstruct bin} Let $x,y \in \{0,1\}^k$ be two binary strings. If $x\not \sim y$, then $\mathscr{D}_{3,k/2+1}(x)\neq \mathscr{D}_{3,\lfloor k/2-1\rfloor+1}(y)$. 
\end{lem}
This easily gives Lemma~\ref{func deck}.
\begin{proof}[Proof of Lemma~\ref{func deck} assuming Lemma~\ref{reconstruct bin}]
Let $f,g \in \N^k$ be such that $x\not \sim y$. It suffices to show that there exists $\pi:\N\to \{0,1\}$ so that for the binary strings $x=\pi\circ f,y = \pi \circ g$, we have $x\not \sim y$.

\vs

Now, for $j \in \N$, let $\pi_j$ be the indicator function for $j$, and $x_j,y_j$ respectively be $\pi_j \circ f,\pi_j\circ g$. As $f \neq g$, we have that there is $j_1\in \N$ so that $x_{j_1}\neq y_{j_1}$. Similarly, as $f\neq g'$, there is $j_2\in \N$ so that $x_{j_2} \neq (y_{j_2})'$. 

\vs

If we also had that $x_{j_1} \neq (y_{j_1})'$, then we would have that $x_{j_1} \not \sim y_{j_1}$ and so we are done. The same applies for $x_{j_2}$ and $y_{j_1}$, thus without loss of generality, we are left to assume that $(y_{j_1})' = x_{j_1} \neq y_{j_1} $ and $y_{j_2} = x_{j_2} \neq (y_{j_2})'$.

\vs

In this case, we can take $\pi = \pi_{j_1} +\pi_{j_2}$ (here addition is done in $\F_2$). Treating binary strings as vectors (in $\F_2$), we have that $x:=\pi\circ f = x_{j_1}+x_{j_2},y:=\pi\circ g = y_{j_1}+y_{j_2}$. We have that $x-y = x_{j_1}-y_{j_1}\neq 0^k$, and $x-y' = x_{j_2}-(y_{j_2})' \neq 0^k$, thus $x\not \sim y$.
\end{proof}

\vs

Given a multiset $M$ of strings with length $s$, for $S\subset [s]$, let $M|_S$ be the multiset of strings with length $|S|$, where the multiplicity of $w$ is $\sum_{a \in M} I(a|_S = w)$.

\begin{lem} \label{get xS} Let $\ell > C, s\ge 2$ and consider $S \subset [C]$ with $|S| < s$. For any $x \in \{0,1\}^k$, given $\mathscr{D}_{s,\ell}(x)$, we are able to reconstruct the multiset $x|_S, x'|_S$.
\begin{proof} Let $E = [|S|]+1 = \{2,\dots,|S|+1\}$. 

\vs

Now, observe that for $i < \min(S)$,  we have 
\begin{align*}\#(g \in D_{\{i\}\cup S}(x)|_E) &= \sum_{j \in \Z}I(j+i >0) (I(x|_{j+S} = g)+I(x'|_{j+S} = g)) \\
&= \#(g \in D_S(x)) - \sum_{j \in \Z}I(0\ge j+i)(I(x|_{j+S} = g)+I(x'|_{j+S} = g)).\\
\end{align*}
\vs
So, if $\min(S) \ge 2$, then $D_{\{1\}\cup S}(x)|_E \setminus D_{\{0\}\cup S}(x)|_E$ gives the multiset we wish to reconstruct (note that $\Delta(\{0\} \cup S) \le \Delta(\{0\}\cup[C]) = C+1 \le \ell$, thus its value is provided by $\mathscr{D}_{s,\ell}$). Meanwhile, if $\min(S) = 1$, we instead use $ D_S(x)\setminus D_{\{0\}\cup S}(x)|_E $.
\end{proof}
\end{lem}

\begin{cor}\label{2 decks} Let $\ell > C$, and $i \in [C]$. For any $x \in \{0,1\}^k$, given $\mathscr{D}_{2,\ell}$, we are able to reconstruct $x(i)+x'(i)$.
\begin{proof} Apply Lemma~\ref{get xS} with $S=\{i\}$.
\end{proof}
\end{cor}

\begin{cor} \label{3 decks}Let $\ell > C$, and distinct $i,j \in [C]\cup [k-C+1,k]$. For any $x \in \{0,1\}^k$, given $\mathscr{D}_{3,\ell}$, we are able to reconstruct $\{x(i)|x(j), x'(i)|x'(j)\}$.
\begin{proof} When $i<j \in [C]$, we may simply apply Lemma~\ref{get xS} with $S = \{i,j\}$. The rest of the cases are achieved through simple manipulation. 

\vs

If $i',j' = k-i+1,k-j+1$, then $x(i')|x(j') = x'(i)|x'(j)$, which will imply the outcome for $i',j'$ will be the same as the outcome for $i,j$. If $i,j = j_0,i_0$, then $x(i)|x(j) = x(j_0)|x(i_0)$. If $i,j = i_0,k-j_0+1$, then $x(i)|x(j) = x(i_0)|x'(j_0)$. \end{proof}\end{cor}

\vs
Finally, we can reconstruct binary strings.
\begin{proof}[Proof of Lemma~\ref{reconstruct bin}]
Let $x\in \{0,1\}^k$ be any binary string, and let $x'$ be the reverse of $x$, so that $x'(i) = x(k-i+1)$ for all $ i \in [k]$. It suffices to show that we can reconstruct $x$ or $x'$ from $\mathscr{D}_{3,\lfloor k/2-1\rfloor +1}(x)$. We first define $I= [\lfloor k/2-1\rfloor]\cup[k-\lfloor k/2-1\rfloor ,k]$ and show that we can reconstruct $x|_I$ up to reversal (i.e., construct a string $y$ such that $y \in \{x|_I,x'|_I\}$).

\vs

For each $i\in I$ where $x(i) = x'(i)$, we have that $x(i) = (x(i)+ x'(i))/2$, which we can reconstruct by Corollary~\ref{2 decks}. Also, $x(i) = x'(i)$ if and only if $x(i) + x'(i) \neq 1$, which by Corollary~\ref{2 decks} we know how to recognize (for $i \in I$).

\vs

Now, let $T$ be the set of indices $i$ where $x(i) \neq x'(i)$. If $T = \{\}$, then it is clear how to fully reconstruct $x|_I$. Otherwise, we fix some arbitrarily chosen $i \in T$. WLOG, we may assume $x(i) = 1$, otherwise we may replace $x$ with $x'$ (which does not matter as we only wish to reconstruct $x$ up to reversal). Under this assumption, we may now reconstruct $x|_I$. Indeed, we can simply apply Corollary~\ref{3 decks} repeatedly for each $j  \in T \setminus \{i\}$ and observe that $x(j) =1 $ if and only if $1|1\in \{x(i)|x(j), x'(i)|x'(j)\}$.

\vs

Now suppose $k$ is odd. In this case, $I = [k]\setminus \{\lfloor k/2\rfloor\}$, thus the above reconstructs $x|_{[k]\setminus \{\lfloor k/2\rfloor\}}$ up to reversal. It then remains to reconstruct $x(k^*) = x'(k^*)$. Observing $D_{\{1\}}(x) = |\Supp(x)|$ simply counts the number of ones in $x$, we can easily determine $x(k^*)$ by seeing if $x$ has more ones than $x|_{[k]\setminus \{k^*\}}$.

\vs

Otherwise $k$ must be even. Now the above reconstructs $x|_{[k]\setminus \{k/2,k/2+1\}}$ up to reversal. As before, we may use $D_{\{1\}}(x)$ to reconstruct $|\Supp(x)|$, and we note that if $|\Supp(x)| -|\Supp(x_I)| \neq 1$ then $x(k/2) = x(k/2+1)= (|\Supp(x)| -|\Supp(x_I)|)/2$ in which case we can easily reconstruct $x$ up to reversal.

\vs

We are left to assume $|\Supp(x)| -|\Supp(x_I)| = 1$. Now if $x|_I = x'|_I$ then the rest of $x$ is symmetric and thus reconstruction is still straight-forward (taking $y \in \{0,1\}^k$ so that $y|_I = x|_I = x'|_I$ and $y(k/2) = 0,y(k/2+1) =1$, it is clear $y \in \{x,x'\}$ as desired). So we assume $x|_I \neq x'|_I$, and let $i$ be the largest index $<k/2$ such that $x(i) \neq x'(i)$. WLOG we assume $x(i) = 1$ otherwise we may replace $x$ with $x'$. We shall find out if $x(i)|x(k/2) = 1|1$ which will fully reconstruct $x$. 

\vs

By definition of the $S$-deck, \begin{align*}
    D_{\{i,k/2\}}(x) &= \{x(i)|x(k/2),x'(i)|x'(k/2)\} \cup \{x(i+1)|x(k/2+1),x'(i+1)|x'(k/2+1)\}\\&\cup \{x(j)|x(j+k/2-i),x'(j)|x'(j+k/2-i):j,j+k/2-i\in I\}.\\
\end{align*}Having reconstructed $x|_I$ (up to reversal), we can calculate the third multiset on the RHS. Also, by our choice of $i$ we can also calculate the second multiset on the RHS. If $i+1 \in I$ then $x(i+1) = x'(i+1) = b\in \{0, 1\}$ in which case the multiset is $\{b|0,b|1\}$, otherwise $i+1 = k/2$ in which case the multiset is $\{0|1,1|0\}$. Finally, we know the multiset on the LHS as $\Delta(\{i,k/2\}) \le k/2 = \lfloor k/2-1\rfloor +1$ (because we are told $\mathscr{D}_{3,\lfloor k/2-1\rfloor}(x)$). 

\vs

It follows that the first multiset on the RHS is completely determined. Observing that $x(i)|x(k/2) = 1|1$ if and only if $1|1$ is an element of $\{x(i)|x(k/2),x'(i)|x'(k/2)\}$ we are done.

\end{proof}

\hide{\section{Conclusions}

\subsection{Generalizations}\label{generalize}

We can easily apply these methods to try reconstructing other classes of graphs.

\subsection{Explicit constant for \texorpdfstring{Theorem~\ref{main}}{Theorem 1}}\label{explicit}
One natural question is how large is the $O(1)$-term in Theorem~\ref{main}? This requires an explicit bound for the function $\gamma$ defined in Section~\ref{mom}. We provide a ``back of the envelope'' calculation for this.

We now define $\gamma_0$ such that $\gamma_0(n) \ge \gamma(n)$ for all $n$. It suffices to find $n$}

\section{Acknowledgements}
The author would like to thank Carla Groenland and Tom Johnston for helpful discussion and for notifying him that this problem was of interest. The author further thanks Carla Groenland for helping improve the readability of the paper. The author would also like to thank Zachary Chase for helpful discussions. Finally, the author thanks Daniel Carter for his feedback on the proof of Lemma~\ref{nondegen} and some comments on the layout of the paper.

\vs

After releasing version 1 of this arXiv preprint, the author thanks Carla Groenland, Tom Johnston, and Daniel Carter for pointing out some typographical errors which have now been corrected.

\appendix
\section{An application of inclusion exclusion}\label{app}

Let $a: \N^s\to  \Bbb{Z}$ be a function with finite support. For $\vec{i} \in \N^s$, let \[\sigma_{\vec{i}}(a):= \sum_{\vec{x} \in \N^s} a(\vec{x}) \prod_{k  \in [s]}x_k^{i_k}.\] Also recall from Section~\ref{mom} our definition of \[\eta_{\vec{j},p}(a) = \sum_{\vec{x} \equiv_p \vec{j} } a(\vec{x})  = \sum_{\vec{x} \in \N^s} a(\vec{x}) \prod_{k \in [s]} I(x_k \equiv j_k \mod p).\] 

\vs

Given a set of functions $\mathcal{F}$, and a function $g$, we will say ``$g$ can be expressed as a $\Z$-linear combination of $\mathcal{F}$ (up to modulo $p$)'' if there exist coefficients $c_f \in \Z$ for $f \in \mathcal{F}$ so that 
\[\sum_{f\in \mathcal{F}} c_f f(a) \equiv g(a) \mod p\]for every function $a:\N^s\to \Z$ with finite support.

\vs

We wish to prove the following.

\begin{prp}
For each prime $p$ and $\vec{j}\in \N^s$, we may express $\eta_{\vec{j},p}(a)$ as a $\Z$-linear combination of $\mathcal{B}_p := \{\sigma_{\vec{i}}(a): \vec{i} \in \{0,\dots,p-1\}^s\}$ (up to modulo $p$).

\begin{proof}We will first define a family of intermediate functions $\mathcal{B}_{\vec{j},p}$. We will use basic number theory to express $\BB_{\vec{j},p}$ as a $\Z$-linear combination of $\mathcal{B}_p$, and then use a $\Z$-linear combination of $\BB_{\vec{j},p}$ to express $\eta_{\vec{j},p}(a)$ via inclusion-exclusion, as desired.

\vs

First, observe that for any $\vec{i} \in \{0,\dots, \ell\}^s, \vec{j} \in \{0,\dots, n\}^s$ we have that $\sigma_{\vec{i},\vec{j}}(a) := \sum_{\vec{x} \in \N^s} a(\vec{x}) \prod_{k \in [s]} (x_k - j_k)^{i_k}$ is in the $\Z$-span of $\mathcal{B}_p$. Indeed, we just use binomial expansion and polynomial multiplication.

\vs

We now wish to construct some indicator functions. We recall Fermat's little theorem which states $a^{p-1} \equiv 1 \mod p$ unless $a \equiv 0 \mod p$. Thus given a set $T \subset [s]$, choosing $\vec{i} = \vec{i}_{T,p}\in \N^s$ so that $i_k = (p-1)I(k \in T)$, we get
\[\sigma_{\vec{i},\vec{j}}(a)  \equiv \sum_{\vec{x}\in \N^s} a(\vec{x}) \prod_{k \in T} I(p \nmid x_k-j_k) \mod p.\]

\vs

Via inclusion-exclusion, ranging over the subsets $T$ of $[s]$ we can get $\eta_{\vec{j},p}(a) \mod p$, as desired. Specifically,
\[\eta_{\vec{j},p}(a) \equiv \sum_{T \subset [s]} (-1)^{|T|}\sigma_{\vec{i}_{T,p},\vec{j}}(a) \mod p.\]
\end{proof}\end{prp}

\hide{Let $a: \N^s\to  \Bbb{Z}$ be a function. For $\vec{i} \in \N^S$, let $\sigma_{\vec{i}}(a):= \sum_{\vec{x} \in \N^s} a(\vec{x}) \prod_{k  \in [s]}x_k^{i_k}$. Also recall from Section~\ref{mom} our definition of $\eta_{\vec{j},p}(a) = \sum_{\vec{x} \equiv_p \vec{j} } a(\vec{x})  = \sum_{\vec{x} \in \N^s} a(\vec{x}) \prod_{k \in [s]} I(x_k \equiv j_k \mod p)$. We shall now show how to express $\eta_{\vec{j},p}(a)\mod p$ for any prime $p$ and any $\vec{j} \in \N^s$ as a linear combination of $\mathcal{B}_p = \{\sigma_{\vec{i}}(a): \vec{i} \in \{0,\dots,p-1\}^s\}$.

\vs

Since $p$ is prime, for any integer $N$ we have that $N^{p-1}\equiv I(N) \mod{p}$, where $I(N)$ is the indicator function for $p \nmid N$. 

\vs

First, observe that for any $\vec{i} \in \{0,\dots, \ell\}^s, \vec{j} \in \{0,\dots, n\}^s$ we have that $\sigma_{\vec{i},\vec{j}}(a) := \sum_{\vec{x} \in \N^s} a(\vec{x}) \prod_{k \in [s]} (x_k - j_k)^{i_k}$ is in the span of $\mathcal{B}_p$. Given a set $T \subset [s]$, choosing $\vec{i} = \vec{i}_{T,p}\in \N^s$ so that $i_k = (p-1)I(k \in T)$, we get
\[\sigma_{\vec{i},\vec{j}}(a)  \equiv \sum_{\vec{x}\in \N^s} a(\vec{x}) \prod_{k \in T} I(p \nmid x_k-j_k) \mod p.\]

\vs

Via inclusion-exclusion, ranging over the subsets $T$ of $[s]$ we can get $\eta_{\vec{j},p}(a) \mod p$, as desired. Specifically,
\[\eta_{\vec{j},p}(a) \equiv \sum_{T \subset [s]} (-1)^{|T|}\sigma_{\vec{i}_{T,p},\vec{j}}(a) \mod p.\]}

\end{document}